\documentclass[aos,natbib]{imsart}

\RequirePackage[OT1]{fontenc}
\RequirePackage{amsthm,amsmath}
\RequirePackage[colorlinks,citecolor=blue,urlcolor=blue]{hyperref}

\usepackage{float}
\restylefloat{figure}

\RequirePackage{amsthm,amsmath}
\usepackage{multirow}
\usepackage[arrow, matrix, curve]{xy}
\usepackage{tikz-cd}
\usepackage{tikz}
\usetikzlibrary{decorations.pathmorphing}
\usepackage{bigdelim}
\RequirePackage{amssymb,amstext}
\RequirePackage{amsfonts}
\RequirePackage{amsbsy}
\RequirePackage{mathrsfs}
\RequirePackage{bm}
\RequirePackage{bbm}
\usepackage{multirow,bigdelim}
\usepackage{MnSymbol}
\usepackage{enumerate}
\usepackage{graphicx}
\usepackage{amsthm}
\usepackage{theoremref}

\usetikzlibrary{positioning,arrows}
\tikzset{
  state/.style={circle,draw,minimum size=6ex},
  arrow/.style={-latex, shorten >=1ex, shorten <=1ex}}

\newcommand{\cP}{\mathcal{P}}
\newcommand{\N}{\mathbb{N}}

\newcommand{\bP}{\mathbb{P}}

\newcommand{\ke}{\mathfrak{e}}

\newcommand{\kP}{\mathfrak{P}}
\newcommand{\cL}{\mathbb{L}}

\newcommand{\cS}{\mathcal{S}}

\newcommand{\kB}{\mathfrak{B}}
\newcommand{\cI}{\mathcal{I}}
\newcommand{\cT}{\mathcal{T}}
\newcommand{\cK}{\mathcal{K}}
\newcommand{\Xn}{X_{[1:n]}}
\newcommand{\Dir}{\textsf{Dir}}

\newtheorem{theorem}{Theorem}[section]
\newtheorem{definition}{Definition}

\newtheorem{lemma}[theorem]{Lemma}
\newtheorem{corollary}[theorem]{Corollary}
\newtheorem{remark}{Remark}
\newtheorem{example}{Example}

\arxiv{arXiv:1710.01552}

\startlocaldefs
\numberwithin{equation}{section}
\theoremstyle{plain}

\endlocaldefs

\begin{document}

\begin{frontmatter}
\title{Bayesian inference for stationary data on finite state spaces 
}
\runtitle{Bayesian inference for stationary data}

\begin{aug}
\author{\fnms{F. Moritz} \snm{von Rohrscheidt}\thanksref{t2}\ead[label=e1]{moe0405@gmx.de}},

\thankstext{t2}{Supported by the DFG (German research foundation); grant 1953}
\runauthor{F.M. von Rohrscheidt}

\affiliation{Ruprecht-Karls-Universit{\"a}t Heidelberg 
}

\address{Fritz Moritz von Rohrscheidt\\
Ruprecht-Karls-Universit{\"a}t Heidelberg\\
Institut f{\"u}r Angewandte Mathematik\\
Im Neuenheimer Feld 205 \\
69120 Heidelberg \\
\printead{e1}\\
\phantom{E-mail:\ } 
}

\end{aug}

\begin{abstract}
In this work the issue of Bayesian inference for stationary data is addressed. Therefor, a parametrization of a statistically appropriate subspace of the shift-ergodic probability measures on a Cartesian product of some finite state space is given using an inverse limit construction. Moreover, an explicit model for the prior is given by taking into account an additional step in the usual stepwise sampling scheme of data. An update to the posterior is defined by exploiting this augmented sample scheme. Thereby, its model-step is updated using a measurement of the empirical distances between the classes of considered models.
\end{abstract}



\end{frontmatter}

\section{Introduction}
During the 20th century and beyond, the theory of stochastic processes has created a huge field of models which study and unify a wide range of real world problems. This field is also the natural habitat for the statistician who infers the data-generating mechanism by means of observed data. Thereby, keeping track of large sample accuracy of inference procedures, inference for processes evolving in discrete time - in contrast to high-frequency approaches to data evolving in continuous time - ought to require stationarity as a minimal assumption on the model. Otherwise, if the data-generating mechanism changes during time, the statistician either has only access to non-extendable data or must incorporate this change as prior knowledge which is often not available.
However, unless additional assumptions on the structure of the data is made, the space of probability measures that could have generated stationary data is overwhelmingly large. Most prominent assumptions on the structure of the data which reduce the space of considered probability measures were studied in depth and constitute of symmetry constraints like exchangeability, partial exchangeability and others; c.f. \cite{kallenberg2006probabilistic}. As the perhaps greatest advocate of the concept of subjectivistic probability theory of the 20th century, Bruno de Finetti has intensively emphasized the idea of exchangeability in his book \cite{de1974theory} as well as in a published conservation with Hugo Hamaker, see \cite{finetti1979hamaker}. The recurring argument is that the data should have attached some judgment which possesses an actual meaning. Roughly speaking, the judgment of exchangeability means that the order of collected data is not necessary information and hence does not affect statistical inference. For the case of binary data \cite{finetti1937} proved the celebrated representation theorem which states that any sequence of random variables is exchangeable if and only if it is representable as the mixture of sequences of independent and identically distributed random variable. Thereby the mixture is understood in form of an integral w.r.t. a measure that is interpreted to represent ones uncertainty in the true data-generating measure. This subjectivistic approach to probability paves the way for an epistemic kind of inference that enables one to probabilistically predict future outcomes based on past observations. This is mainly due to the fact that exchangeable data, unlike independent data, are in general not uncorrelated, c.f. \cite{kingman1978uses}. 

Subsequently de Finetti's theorem was generalized in many respects. \cite{hewitt1955symmetric} generalized the underlying state space of the data to Polish spaces. \cite{aldous1981representations} introduced the concept of exchangeability to arrays of random numbers in order to prove a conjecture of \cite{dawid1978extendibility} about the representability of spherically invariant random matrices. Another important generalization was given with respect to sequences of random variables whose joint law is only invariant with respect to a suitable subfamily of the symmetric group. This subfamily can be thought of appropriate block-transformations of the coordinates which do not affect the number of transitions the process makes between the states. De Finetti already hinted to the fact that those sequences are representable as mixtures of Markov chains and \cite{diaconis1980finetti} gave the rigorous proof. This is based on the idea of the exchangeability of the successor blocks of one recurrent state, showing that recurrence is a necessary condition for such mixtures. However, recurrence is certainly implied by some overruling statistical judgment, namely stationarity. Hence, it not surprising that mixtures of Markov chains can be studied from a statistical viewpoint which was done in \cite{freedman1962invariants}. Freedman showed even more; any stationary probability measure on a Cartesian product of some countable state space which is summarized by a statistic possessing $S$-structure is representable as an integral mixture of shift-ergodic probabilities being also summarized by this statistic. Thereby, $S$-structure is a certain equivalence-invariance under stacking the data. Since transition counts have $S$-structure and since Markov measures are shown to equal the shift-ergodic probabilities being summarized by transition counts, \cite{freedman1962invariants} obtained the de Finetti theorem for Markov chains. Related ideas dealing with statistical mixtures in an extended way can be found in \cite{lauritzen1988extremal}, see also \cite{schervish1995theory} for a well accessible source. 

More recently, the Italian school took up the issue of mixing up probability measures in further depth. \cite{fortini2002mixtures} confirm a conjecture of de Finetti which says that a recurrent probability on a Cartesian product of some countable state space is a mixture of Markov measures if and only if the associated matrix of successor-states is exchangeable within rows. In this work they also considered the more challenging case of an uncountable state space by using suitable weak approximations of the appearing kernels. \cite{epifani2002characterization} extends related ideas to semi-Markov processes. The works \cite{fortini2000exchangeability,fortini2014predictive} characterize mixtures of probabilities using predictive distributions. Moreover, in \cite{fortini2014predictive} it is proven that the push-forward of a prior under an appropriate parametrization of the Markov measures samples the rows of the associated stochastic matrix independently if and only if transition counts are predictively sufficient.
In \cite{fortini2012hierarchical,fortini2012predictive} the authors shed light on the construction of priors obtained as weak limits of predictive schemes similar to the one by \cite{blackwell1973ferguson} yielding the Dirichlet process prior. These constructions yield priors concentrated on the space of Markov measures using reinforced urn processes. In \cite{diaconis2006bayesian} the construction of conjugate priors for reversible Markov measures is given and \cite{bacallado2011bayesian} extend these ideas to higher-order Markov measures and provide procedures for testing the order. Besides of that, \cite{lijoi2007bayesian} give a prove of a Doob-type consistency result for stationary models which, as one would expect, rely on martingale convergence arguments. It is interesting to see that theory can provide such a deep result which, at the same time, lacks applications in form of highly general models. This lack can best be explained by the overwhelming richness of the space of shift-ergodic measures that serves as the parameter space. Roughly, this richness relies on the fact that a shift-ergodic probability measures in general does not allow for a parametrization in a sequence of stochastic kernels which is ''Noetherian``of any degree, in the sense that it eventually remains constant. \cite{berti2000parametric,berti2003bayesian} underline this arguing that the appropriate mixing measure (occasionally called the Maitra measure in honor of Ashok Maitra, see \cite{maitra1977integral}) can be obtained as the law of the weak limit of a suitably chosen sequence of empirical measures increasing in the degree of dependence. This can also be reflected by considering invariance of probability measures with respect to transformations consisting of interchanging blocks (c.f. \cite[Proposition 27]{diaconis1980finetti}) whose initial and terminal strings of states, which have to coincide within the blocks to be admissible, increase continuously in their lengths.
From a rather philosophical point of view Bayesian inference for stationary data is quite challenging because one has to express one's uncertainty in the structure of the data itself and has to learn for it from \textit{one single} series of experiments.

Besides calling back the issue to mind, the aim of the present work is twofold. On the one hand it gives a parametrization of of a statistically suitable subspace of the shift-ergodic measures on the Cartesian product of some finite state space. On the other it provides an explicit model for a prior exhausting the arising parameter space as well as for its update to the posterior. For the former, above-mentioned idea of \cite{berti2000parametric} motivates the construction of the parameter space in terms of an appropriate limit. The model itself is then constructed by extending the usual step-wise sampling scheme, i.e. first sample the data-generating mechanism and then sample the data from it, by adding an additional step incorporating the degree of dependence among the data. The considered model can be argued to be non-informative in a specific way.

The paper is organized as follows. In section 2 the necessary preliminaries and notations are given. Section 3 deals with the parametrization of the shift-ergodic measures. Section 4 provides the explicit Bayesian model for stationary data. Section 5 concludes with an outlook to and a discussion of further issues.


\section{Preliminaries}

Let $S$ be a finite space with cardinality $\#S=s$. As usual, equip the Cartesian product $S^{\N}=\prod_{n\in\N}S$ with the product 
$\sigma$-field $\cS$ generated by the semi-algebra consisting of all cylinder sets of the form 
\begin{align*}
[x_1,x_2,\dotsc,x_k]=\left\{x\in S^{\N}:\pi_j(x)=x_j; j=1,\dotsc,k; k\in\N\right\},
\end{align*}
where $\pi_j$ denotes the projection onto the $j$-th coordinate.
Denote by $\cP(T)$ the space of all probability measures on a Polish space $(T,\cT)$. It is well known, see e.g. \cite{kechris1995descriptive}, that $\cP(T)$ itself becomes a Polish space using the appropriate Borel $\sigma$-field $\kB_{\cP}$ induced by weak convergence of measures. Equivalently, $\kB_{\cP}$ is the smallest $\sigma$-field which makes the mappings $\mu_T: P\mapsto P(T)$ measurable for all $T\in\cT$, c.f. \cite{ghosh2003bayesnonp}. Let $\theta:\cP(S^{\N})\rightarrow \cP(S^{\N})$ be the shift operator on the space of probability measures defined through
\begin{align*}
P([x_1,\dotsc,x_k]) \stackrel{\theta}{\longmapsto} 
P(\bigcup_{j\in S} [j,x_1,x_2,\dotsc,x_{k}])
\end{align*}

for all cylinder sets, where $\tilde{\theta}$ denotes the shift on $S^{\N}$. Furthermore, call a probability $P\in\cP(S^{\N})$ stationary if $\theta P=P$ holds. By Kolmogoroff's extension theorem, that is equal to 

\begin{align*}
\theta P([x_1,\dotsc,x_k])=P(\bigcup_{j\in S} [x_1,x_2,\dotsc,x_{k},j]),
\end{align*}

for all cylinder sets, which amounts to an expression of stationarity in terms of a symmetry constraint on $P$. Moreover, call a set $I\in\cS$ shift-invariant if it holds $\theta I\triangle I=\emptyset$, where $\triangle$ denotes the symmetric difference. It is easy to see that the family of invariant sets forms a $\sigma$-field $\cI$, known as the shift-invariant $\sigma$-field. A stationary probability $P$ is called shift-ergodic if it holds $P(I)P(I^c)=0$ for all shift-invariant sets $I$. Denote by $\cP^e(S^{\N})$ the space of all ergodic probability measures. By the ergodic decomposition theorem, c.f. \cite{varadarajan1963groups,maitra1977integral}, any stationary probability is uniquely representable as an integral mixture of ergodic probabilities, i.e. for any set $A\in\cS$ it holds
\begin{align*}\label{EDC}\tag{EDC}
P(A)=\int_{\cP^e(S^{\N})} Q(A) \Pi(dQ).
\end{align*}
Thereby, the distribution $\Pi\in\cP\left(\cP^e(S^{\N})\right)$, interpreted to express one's uncertainty, is uniquely determined and serves as the prior distribution in Bayesian statistics. 

Let $(\Omega,\mathcal{A},\bP)$ some probability space and $\left(X:\Omega\rightarrow S^{\N}\right)=\left(X_n:\Omega\rightarrow S\right)_{n\in\N}$ a discrete-time stochastic process with state space $S$. The law of a random object $Y$ with respect to $\bP$ shall be written as $\mathbb{L}[Y]$.
By virtue of \cite{berti2000parametric}, $\Pi$ can be thought of as the weak limit of empirical measures increasing by their degree of dependence $N\in\N_0$ which are given as
\begin{align*}
\ke_n^{(N)}(\cdot)=\frac{1}{n-N}\sum\limits_{k=1}^{n-N} \delta_{(X_j,\dotsc,X_{j+N})}(\cdot), \textsf{\qquad}n>N. \label{E} \tag{E}
\end{align*}
Precisely, they show that the sequence $(\ke_n^{(\lfloor n/2\rfloor)}\times \alpha)_{n\in \N_0}\in\cP^{\N}(S^{\N})$ consisting of arbitrarily augmented empirical measures (meaning that $\alpha\in\cP(S^{\N})$ makes $\ke_n^{(\lfloor n/2\rfloor)}$ a proper probability measure on the sequence state space) converges weakly to the mixing measure $\Pi$ as $n$ tends to infinity for any reasonable choice of $\alpha$ and for $\cL[X]$-almost all data. So, asymptotically speaking, the extension $\alpha$ does not play any role in this construction of $\Pi$. Notice that the choice for $N$ as the integer part of $n/2$ is quite arbitrary such that above sequence can be replaced by any other of the form $(\ke_n^{(N_n)}\times \alpha)_{n\in \N_0}$ such that $N_n\in \N$ and $n/N_n\stackrel{n \rightarrow \infty}{\longrightarrow}\infty$. This weak limit construction reflects the fact that the shift-ergodic probability measures are summarized by higher order transition counts are higher order Markov chains, which can easily be obtained as a slight generalization of \cite[Theorem 2]{freedman1962invariants}. Additional prior knowledge on the data in form of some additional symmetry judgment lets $\Pi$ concentrate on a proper subspace of $\cP^e(S^{\N})$. For instance, if the data is judged to be exchangeable, the support of $\Pi$ becomes the space of i.i.d.-measures.
From the viewpoint of functional analysis, $\Pi$ is given as the representing measure of the Choquet simplex $\cP(S^{\N})$, see e.g. \cite{phelps2001lectures} for details.

By a stochastic kernel from one measurable space $(S,\cS)$ to another $(T,\cT)$ it is meant a measurable mapping $\kappa:S \rightarrow \cP(T)$.


\section{Parametrization of shift-ergodic measures}

The present section is devoted to a decent parametrization of a statistically reasonable subset of the the shift-ergodic measures. The need for such a parametrization arises from the ergodic decomposition \eqref{EDC} which pinpoints the space of shift-ergodic measures as the appropriate parameter space for the issue of Bayesian statistics for stationary data. Under additional judgments of the data the parameter space reduces considerably, c.f. \cite{diaconis1980finetti,diaconis1985quantifying,diaconis1987dozen}. However, besides taking values in a finite state space, such an additional constraint on the data is not assumed here. 

Roughly, the idea of the mentioned parametrization can be described as follows. The Ionescu-Tulcea theorem (see e.g. \cite{klenke2013probability}) is exploited in order to parametrize a probability measure on a Cartesian product of the state space in terms of a sequence of kernels. In the case of a Markov measure, this sequence of kernels is basically constant and the determining kernel of degree one can be thought of as a stochastic matrix. In full analogy, in the general case, one can think of the higher-dimensional kernels as multi-linear mappings which are introduced in this section as so called stochastic tensors. By flattening these tensors according to some reasonable convention, see e.g. \cite{landsberg2012tensors}, one can also think of these tensors as a stochastic matrix determining a Markov chain evolving in a higher-dimensional space. As in the case of a usual Markov chain it is then shown that stationarity of the parametrized probability measure together with a statistically reasonable constraint on the parametrizing tensors implies the existence of a family of non-linear mappings such that some stochastic tensor is always given as the image of the one of next higher dimension under the associated mapping. The parameter space is then defined as the inverse limit alongside the sequence of kernels and the obtained family of mappings. Notice that inverse limit extensions have been used successfully in stochastics and statistics since Kolmogoroff's extension theorem. They always give rise to the existence of some abstract limiting object which, loosely speaking, can be seen to form a kind of closure extending finite procedures. More recent work on statistics using the idea of inverse limits can be found in \cite{orbanz2010conjugate,orbanz2011projective}. In the first paper, the author generalizes the idea of a natural conjugate prior for exponential families to the nonparametric case by pinning conjugate priors for all finite dimensional marginals and subsequently extending it by an inverse limit argument. The second paper deals with the construction of random probability measures similar to but in greater generality as the construction of the Dirichlet process.

Let $S=\left\{1,\dotsc,s\right\}$ be the state space of the data and let be given a sequence of kernels 
\begin{align*}
\left(\kappa_N:S^N\times \textsf{Pot}(S) \rightarrow [0,1] \right)_{N\in\N_0}
\end{align*}
 increasing by the degree of dependence $N$. Thereby the convention is used that $\kappa_0$ is an initial probability measure on the state space. By the Ionescu-Tulcea theorem there is one and only one probability measure $Q\in\cP(S^{\N})$ on the Cartesian product of the state space such that $Q=\bigotimes_{m\in\N_0}\kappa_m$. In accordance to the Markov case the stochastic kernels are written as $\kappa_N=\left(p_{s_1\cdots s_{N}}^{r_1\cdots r_N}\right)_{r_i,s_j,\in S;i,j=1,\dotsc,N}$, where the contra-variant index $(r_1,\dotsc,r_N)$ denotes the position right before and the co-variant index $(s_1,\dotsc,s_N)$ the position right after a jump of an $N$-dimensional process evolving in discrete time. Thus, in order to model a process of data possessing dependence of length $N$, it is assumed to hold
\begin{align*}
p_{s_1\cdots s_{N}}^{r_1\cdots r_N}
\begin{cases}
=0, & \textsf{if }  (r_1,\dotsc,r_{N-1},r_N)\neq(r_2,\dotsc,r_N,s_{N})\\
\in [0,1], & \textsf{otherwise}
\end{cases}.
\end{align*}

Letting for $N\geq 1$ denote $\cK_N\subset [0,1]^{s^{2N}}$ the space of stochastic tensors of degree $N$, equip $\cK_N$ with the topology of coordinate-wise convergence $\cT_N$ in order to obtain a topological space $\left(\cK_N,\cT_N\right)$ and a measurable space $\left(\cK_N,\kB_N\right)$, respectively, where $\kB_N$ is the $\cT_N$-associated Borel $\sigma$-field.
The subsequent result establishes the meaning of stationarity in terms of the parametrizing sequence of kernels.

\begin{theorem}\thlabel{thm}
Let $\left(\kappa_N:S^N\times \textsf{Pot}(S)\rightarrow (0,1)\right)_{N\in\N_0}$ be a sequence of positive stochastic tensors and $Q=\bigotimes_{N\in\N_0}\kappa_N\in\cP(S^{\N})$. Then, stationarity of $Q$ implies the unique existence of a family of continuous mappings 
\begin{align*}
\left(\phi_N:\cK_N\rightarrow \cK_{N-1}\right)_{N\in\N},
\end{align*}
 such that $\phi_N(\kappa_N)=\kappa_{N-1}$.
Moreover, if $Q$ is stationary it is also shift-ergodic.
\end{theorem}
\begin{proof}
Let $N$ be a fixed but arbitrarily chosen positive integer, $r:=s^N$ and $\kappa_N$ a positive stochastic tensor encoding dependence of length $N$. Applying a suitable flattening $\tilde{\cdot}$, $\kappa_N$ can be thought of as a stochastic $(r\times r)$-matrix $\tilde{\kappa}_N$ governing a Markov chain on the $N$-th higher shift space $S^{[N]}$, c.f. \cite{lind1995introduction}. Now, positivity of $\kappa_N$ clearly implies positive recurrence and aperiodicity of $\tilde{\kappa}_N$. From the theory of Markov chains, c.f. \cite{freedman1983markov}, it is well known that in this case there is a unique stochastic row vector $v$ of length $r$ such that $v\tilde{\kappa}_N=v$. Thus, in essence, there is a mapping 
\begin{align*}
\kappa_N \stackrel{\gamma_N}{\longmapsto} v,
\end{align*}
which is continuous by virtue of \cite{schweitzer1968perturbation}.
Furthermore, define the continuous mapping $\zeta_N$ by
\begin{align}
v=\left(v_{t}\right)_{t\in S^{r}} \stackrel{\zeta_N}{\longmapsto} \left(\frac{v_{t}}{\sum\limits_{j\in S} v_{(j,\pi_{[1:r-1]}\left(t\right))}}\right)_{t\in S^{r}}=:\left(v_{\pi_{[2:r]}(t)}^{\pi_{[1:r-1]}(t)}\right)_{t\in S^{r}},\label{S} \tag{Stat}
\end{align}
where $\pi_{[k:l]}$ denotes the projection onto the coordinates with indexes $k\leq i\leq l$. Now, let $\phi_N:=\zeta_N\circ \gamma_N$, $\kappa_{N-1}:=\phi_N(\kappa_N)$ and extending $\kappa_{N-1}$ by an appropriate amount of zero-entries. In order to close the proof, let $A,B\in\cS$ be two measurable sets with $Q(A),Q(B)>0$, where $Q$ is as in the claim. By positivity of the parametrizing sequence $\kappa$ these sets must contain proper subsets $A'=\theta^{-k_1}[c_1]\cap \cdots \cap \theta^{-k_n}[c_n]$ and $B'=\theta^{-l_1}[d_1]\cap \cdots \cap \theta^{-l_m}[d_m]$ with positive integers $k_n,l_m,n,m$ and states $c_1,\dotsc,c_n,d_1,\dotsc,d_m$. But then, $\theta^{-(l_m+q)}A'\cap B'$ is again of the same form as $A'$ and $B'$ for any $q\in\N_0$. Thus, again by positivity of $\kappa$, it follows $Q\left(\theta^{-(l_m+q)}A'\cap B'\right)>0$ which is well known to be equivalent to ergodicity of $Q$.
\end{proof}

\begin{remark}
\begin{enumerate}[(i)]
\item The proof of \thref{thm} shows that the mapping $\phi_N$ basically amounts to the mapping which maps a decent stochastic matrix onto its Perron-projection.
	\item The construction in \eqref{S} uses stationarity in the following way. Letting $t=(t_1,\dotsc, t_r)$ with $r$ and $Q$ as above one has that 
	\begin{align*}
	p_{t_2\cdots t_r}^{t_1\cdots t_{r-1}}=\frac{v_{(t_1,\dotsc, t_r)}}{\sum\limits_{j\in S}v_{(j,t_1,\dotsc,t_{r-1}}} \Leftrightarrow 
	& \left(\sum\limits_{j\in S} Q ([j,t_1,\dotsc,t_{r-1}]) \right)p_{t_2\cdots t_r}^{t_1\cdots t_{r-1}}\\
	& =Q([t_1,\dotsc,t_{r}]),
	\end{align*}
	which holds true since it is well known that stationarity of $Q$ is equivalent to 
	\begin{align}
	\sum\limits_{j\in S} Q([j,t_1,\dotsc,t_{r-1}])= Q([t_1,\dotsc,t_{r-1}])=\sum\limits_{j\in S} Q([t_1,\dotsc,t_{r-1},j]), \label{sym} \tag{*}
	\end{align}
	for all possible cylinder sets, where the second identity holds by virtue of Kolmogorov's extension theorem. Notice that \eqref{sym} states stationarity, above defined as an invariant constraint, in form of a symmetry condition.
	\item \thref{thm} makes obvious the well known fact that the stationary (ergodic) probability measures on a Cartesian product is (weakly) nowhere-dense in the space of all (the extreme points of all) probability measures on that sequence space.
	\item Continuity of the mappings $\phi_N$ will below be used to induce a measurable structure on the parameter space.
\end{enumerate}
\end{remark}

 The following example establishes the mappings $\phi_1,\phi_2,\phi_3$ in the case of binary data.

\begin{example}
Let $S=\{0,1\}$ and some sequence $\kappa=\left(\kappa_N\right)_{N\in\N_0}$ of positive stochastic tensors be given such that $\kappa_N$ essentially equals $\kappa_3$ for all $N>3$. Thus, $Q=\bigotimes_{N\in\N_0}\kappa_N$ is a $3$-Markov measure. Then, stationarity of $Q$ implies the following set of non-linear equations.
\begin{align}
p_{s_1}p_{s_2}^{s_1}p_{s_2s_3}^{s_1s_2}p_{s_2s_3s_4}^{s_1s_2s_3} 
-p_{s_2}p_{s_3}^{s_2}p_{s_3s_4}^{s_2s_3}=0; s_1,s_2,s_3,s_4\in\{0,1\}.\label{N}\tag{N}
\end{align}
Solving \eqref{N} simultaneously for $\kappa_0,\kappa_1,\kappa_2$, under the stochasticity constraint of the tensors, gives the following solutions in terms of $\kappa_3$.
\begin{align*}
&\kappa_2(\kappa_3)
=\begin{pmatrix}
	p_{00}^{00}(\kappa_3)&p_{01}^{00}(\kappa_3)&0&0\\
	0&0&p_{10}^{01}(\kappa_3)&p_{11}^{01}(\kappa_3)\\
	p_{00}^{10}(\kappa_3)&p_{01}^{10}(\kappa_3)&0&0\\
	0&0&p_{10}^{11}(\kappa_3)&p_{11}^{11}(\kappa_3)
\end{pmatrix}\\
&=\begin{pmatrix}
	\frac{p_{000}^{100}}{p_{000}^{100}+p_{001}^{000}} & \frac{p_{001}^{000}}{p_{000}^{100}+p_{001}^{000}} &0&0\\
	0&0& \frac{p_{010}^{101}+p_{100}^{110}(p_{010}^{001}-p_{010}^{101})}{1+(p_{010}^{101}-p_{010}^{001})(p_{100}^{010}-p_{100}^{110})} & \frac{p_{011}^{101}-p_{100}^{010}(p_{010}^{001}-p_{010}^{101})}{1+(p_{010}^{101}-p_{010}^{001})(p_{100}^{010}-p_{100}^{110})} \\
	\frac{p_{100}^{110}+p_{010}^{101}(p_{100}^{010}-p_{100}^{110})}{1+(p_{010}^{101}-p_{010}^{001})(p_{100}^{010}-p_{100}^{110})} & \frac{p_{101}^{110}-p_{010}^{001}(p_{100}^{010}-p_{100}^{110})}{1+(p_{010}^{101}-p_{010}^{001})(p_{100}^{010}-p_{100}^{110})} &0&0 \\
	0&0& \frac{p_{110}^{111}}{p_{110}^{111}+p_{111}^{011}} & \frac{p_{111}^{011}}{p_{110}^{111}+p_{111}^{011}} \\
\end{pmatrix},
\end{align*}
where it remains to check that all entries of the stochastic tensor are actually probabilities. Indeed, that is true according to the following. First note that, by positivity of the stochastic tensor $\kappa_3$, all the denominators of the entries of $\kappa_2(\kappa_3)$ are strictly positive. Hence, e.g. for $p_{00}^{10}(\kappa_3)$ it holds
\begin{align*}
\frac{p_{100}^{110}+p_{010}^{101}(p_{100}^{010}-p_{100}^{110})}{1+(p_{010}^{101}-p_{010}^{001})(p_{100}^{010}-p_{100}^{110})}\leq 0& \Leftrightarrow p_{100}^{110}+p_{010}^{101}(p_{100}^{010}-p_{100}^{110}) \leq 0 \\
&\Leftrightarrow p_{100}^{110}p_{011}^{101}+p_{010}^{101}p_{100}^{010}\leq 0,
\end{align*}
which clearly contradicted the positivity of $\kappa_3$. The claim for  $p_{00}^{10}(\kappa_3)$ follows since the rows obviously sum up to $1$. The other entries follow the same reasoning. The solutions for $\kappa_1$ and $\kappa_0$ can be obtained as
\begin{align*}
\kappa_1(\kappa_2(\kappa_3))&=
\begin{pmatrix}
	p_{0}^{0}(\kappa_2(\kappa_3))&p_{1}^{0}(\kappa_2(\kappa_3))\\
	p_{0}^{1}(\kappa_2(\kappa_3))&p_{1}^{1}(\kappa_2(\kappa_3))
\end{pmatrix}\\
&=\begin{pmatrix}
	\frac{p_{00}^{10}(\kappa_3)}{p_{01}^{00}(\kappa_3)+p_{00}^{10}(\kappa_3)} & \frac{p_{01}^{00}(\kappa_3)}{p_{01}^{00}(\kappa_3)+p_{00}^{10}(\kappa_3)}\\
	\frac{p_{10}^{11}(\kappa_3)}{p_{10}^{11}(\kappa_3)+p_{11}^{01}(\kappa_3)} & \frac{p_{11}^{01}(\kappa_3)}{p_{10}^{11}(\kappa_3)+p_{01}^{11}(\kappa_3)}
\end{pmatrix}\\
&=\begin{pmatrix}
	\frac{C_1}{1+C_1}&\frac{1}{1+C_1}\\
	\frac{1}{1+C_2}&\frac{C_2}{1+C_2}
\end{pmatrix}
\end{align*}
and 
\begin{align*}
\kappa_0(\kappa_3)&=\left(p_{0}^{}(\kappa_1(\kappa_2(\kappa_3))),p_{1}^{}(\kappa_1(\kappa_2(\kappa_3)))\right)\\
&=\left(\frac{p_{0}^{1}(\kappa_2(\kappa_3))}{p_{0}^{1}(\kappa_2(\kappa_3))+p_{1}^{0}(\kappa_2(\kappa_3))}\frac{p_{1}^{0}(\kappa_2(\kappa_3))}{p_{0}^{1}(\kappa_2(\kappa_3))+p_{1}^{0}(\kappa_2(\kappa_3))}\right)\\
&=\left(\frac{1}{1+\frac{1+C_2}{1+C_1}},\frac{1}{1+\frac{1+C_1}{1+C_2}}\right),
\end{align*}
where 
\begin{align*}
&C_1=\frac{p_{000}^{100}+p_{001}^{000}}{1+\left(p_{010}^{101}-p_{010}^{001}\right)\left(p_{100}^{010}-p_{100}^{110}\right)}\times\frac{p_{100}^{110}+p_{010}^{101}\left(p_{100}^{010}-p_{100}^{110}\right)}{p_{001}^{000}}> 0\\
&C_2=\frac{p_{110}^{111}+p_{111}^{011}}{1+\left(p_{010}^{101}-p_{010}^{001}\right)\left(p_{100}^{010}-p_{100}^{110}\right)}\times\frac{p_{011}^{101}-p_{100}^{010}\left(p_{010}^{001}-p_{010}^{101}\right)}{p_{110}^{111}}>0.
\end{align*}

Notice that these unique solutions also could have been obtained by applying the machinery mentioned in above proof instead of solving \eqref{N} simultaneously.
\end{example}

The example of binary data is not only useful to illustrate how to obtain lower-dimensional kernels from higher-dimensional ones but also for motivating the subsequent definition of a parameter space of the positive shift-ergodic measures. The following pictures show the embedding of the binary shift-ergodic $1$-Markov measures in the space of all binary $1$-Markov measures.
\begin{figure}[H]
	\begin{minipage}{0.45\textwidth} 
	\includegraphics[width=\textwidth]{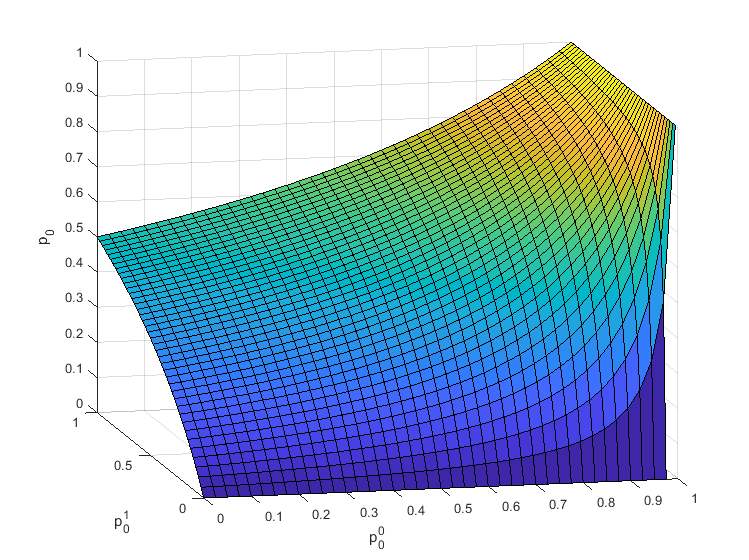}
	\end{minipage}
	\hfill
	\begin{minipage}{0.45\textwidth}
	\includegraphics[width=\textwidth]{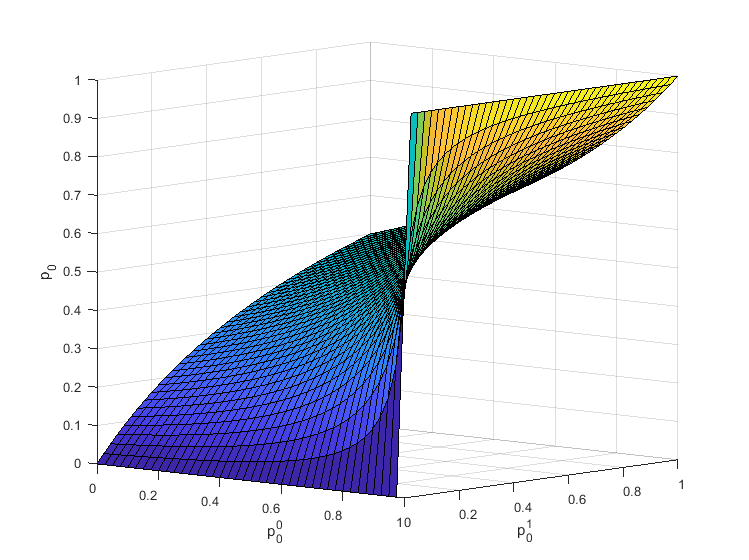} 
	\end{minipage}
\end{figure}

\begin{figure}[H]
	\begin{minipage}{0.45\textwidth} 
	\includegraphics[width=\textwidth]{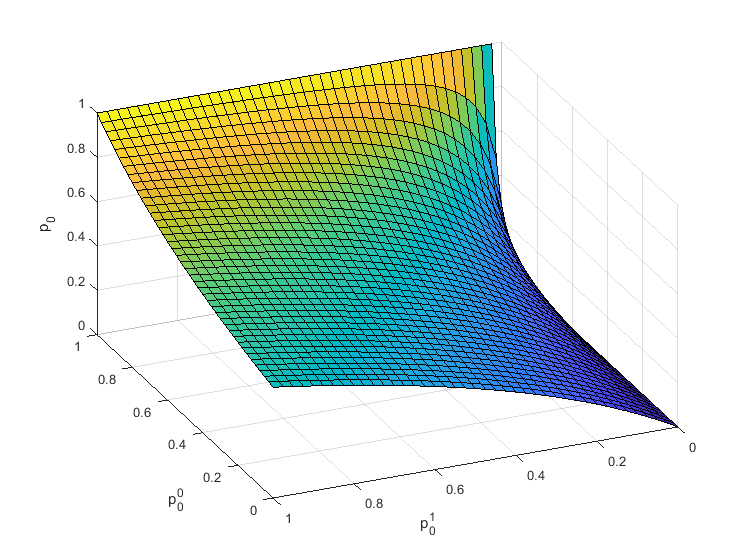}
	\end{minipage}
	\hfill
	\begin{minipage}{0.45\textwidth}
	\includegraphics[width=\textwidth]{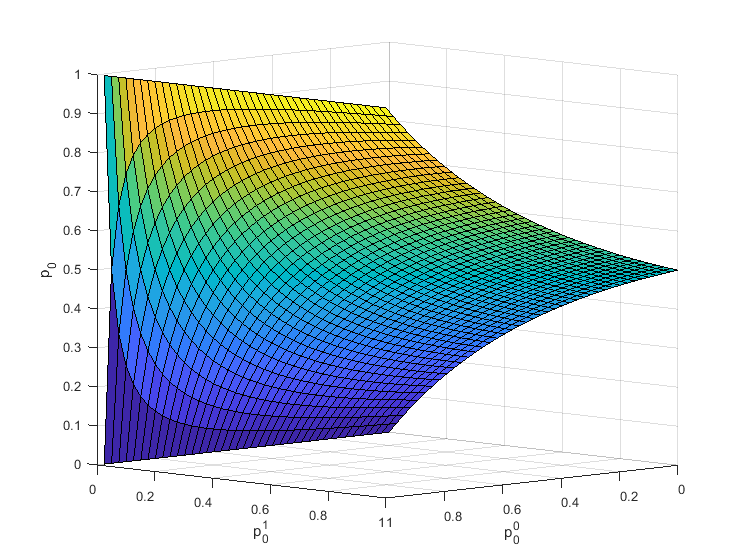} 
	\end{minipage}
	\caption{Embedding of shift-ergodic Markov measures in all Markov measures; the binary case.}
\end{figure}

Each point of the unit cube parametrizes a $1$-Markov measure $P\in\cP(\{0,1\}^{\N})$ according to the definition
\begin{align*}
P([s_1,\dotsc,s_m]):=p_{s_1}p_{s_2}^{s_1}\cdots p_{s_m}^{s_{m-1}},
\end{align*}
for all cylinder sets $[s_1,\dotsc,s_m]$. Thereby $\kappa_0=(p_0,p_1)$ denotes the initial stochastic tensor and $\kappa_2=
\begin{pmatrix}p_{0}^{0}&p_{1}^{0}\\p_{0}^{1}&p_{1}^{1}\end{pmatrix}$ the tensor of next higher dimension.
  However, only the points on the surface depict the shift-ergodic $1$-Markov measures reflecting the non-linear mapping $\cK_1 \stackrel{\phi_1}{\longmapsto}\cK_0$.\\
	
In addition, the Bernoulli measures, can be embedded into the picture by taking
\begin{align*}
\begin{pmatrix}p_{0}^{0}&p_{1}^{0}\\p_{0}^{1}&p_{1}^{1}\end{pmatrix}
:=\begin{pmatrix}p_{0}^{}&p_{1}^{}\\p_{0}^{}&p_{1}^{}\end{pmatrix}.
\end{align*}

\begin{figure}[H]
	\begin{minipage}{0.45\textwidth} 
	\includegraphics[width=\textwidth]{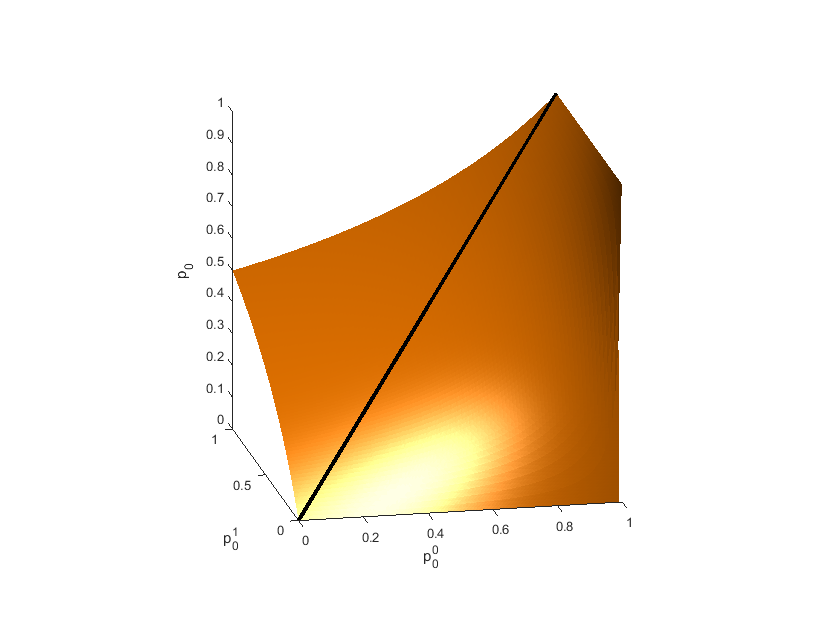}
	\end{minipage}
	\hfill
	\begin{minipage}{0.45\textwidth}
	\includegraphics[width=\textwidth]{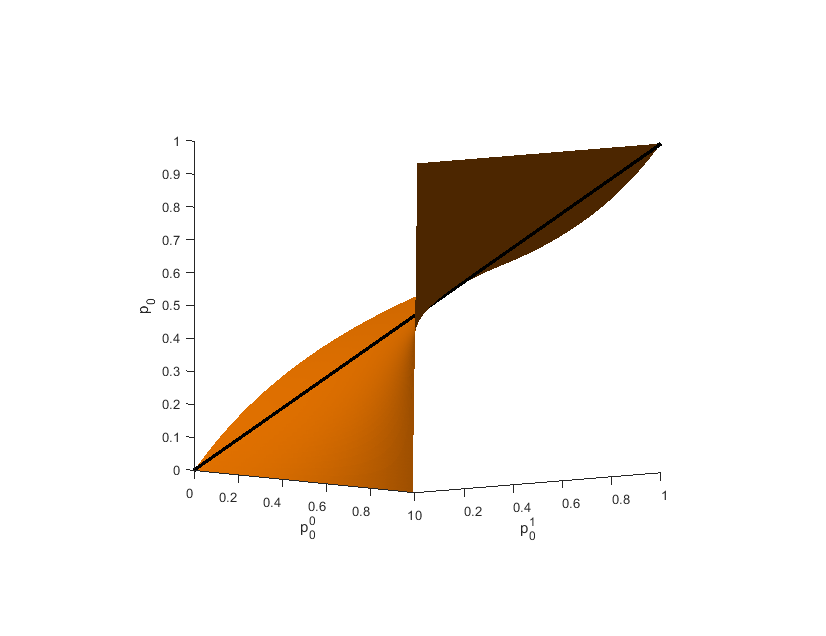}
	\end{minipage}
\end{figure}

\begin{figure}[H]
	\begin{minipage}{0.45\textwidth} 
	\includegraphics[width=\textwidth]{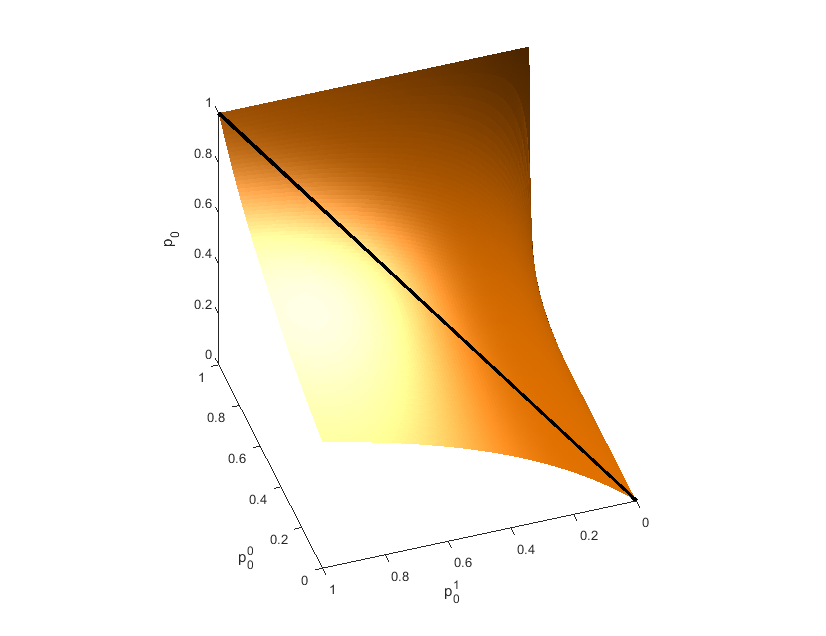}
	\end{minipage}
	\hfill
	\begin{minipage}{0.45\textwidth}
	\includegraphics[width=\textwidth]{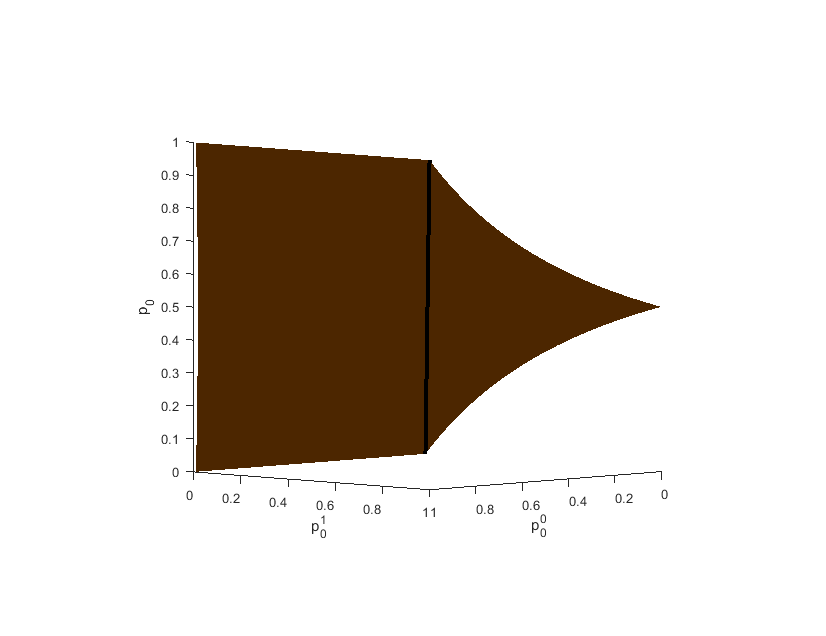}
	\label{fig:2} 
	\end{minipage}
	\caption{Embedding of i.i.d.-measures in shift-ergodic Markov measures; the binary case.}
\end{figure}

In a certain way, this embedding will now be generalized to infinite dimensions. However, this cannot be done in a neat graphical way as above since in higher dimensions or for state spaces with more than two elements, respectively, the involved simplices become too large. Thus, in the following the necessary construction is given in an algebraic manner.

Therefor, let $\cK=\prod_{N\in\N_0} \cK_N$ the space of sequences of stochastic tensors and equipped with the product topology and denote by $\cK(+)\subset \cK$ the dense subset of sequences of positive stochastic tensors. Now, out of the family of continuous mappings $(\phi_N)_{N\in\N}$  appearing in \thref{thm}, build another family of continuous mappings $\psi=\left(\psi_{M,N}:\cK_N(+)\rightarrow \cK_M(+)\right)_{M\leq N;M,N\in\N_0}$ by the construction
\begin{align*}
\psi_{M,N}:=\begin{cases} \phi_{M+1}\circ\cdots\circ\phi_N, & \textsf{if } M\neq N\\ 
                          \textsf{id}_{\cK_M(+)}, &  \textsf{otherwise}
\end{cases}.
\end{align*}
For any sequence $a=(a_n)_{n\in\N}$ denote by $\pi=\left(\pi_n:a\mapsto a_n\right)_{n\in\N_0}$ the family of coordinate projections. The next definition establishes the notion of a parameter space for the space of positive shift-ergodic measures.

\begin{definition}
Let $\cK(+)$ and $\psi$ be as above and define the parameter space $\kP$ to be the inverse limit alongside the families $\left(\cK_N\right)_{N\in\N_0}$ and $\left(\psi_{M,N}\right)_{M,N\in\N_0}$. That is
\begin{align*}
\kP=\left\{\kappa\in\cK(+): \pi_M(\kappa)=\psi_{M,N}\circ\pi_N(\kappa); M,N\in \N_0, M\leq N\right\}.
\end{align*}
\end{definition}

Put another way, the parameter space consists of all sequences of positive kernels increasing by their degree of dependence whose coordinates stick together through the family of mappings $\phi$ obtained in \thref{thm} meaning that the diagram 
\begin{center}
\begin{tikzcd}[column sep=4em]
 & \arrow{ld}{\psi_{N}}\kP \arrow{dr}{\psi_M} \\
 \cK_N\arrow{rr}{\psi_{M,N}} && \cK_M
\end{tikzcd}
\end{center}
commutes, where $\psi_N$ denotes the restriction of $\pi_N$ to $\kP$.
Of course, all ergodic probability measures carrying only finite degree of dependence can be encountered in $\kP$ by a similar embedding argument as Figure 2 implies. In particular these are measures which are obtained by solving a finite system of non-linear equations similar to \eqref{N}. However, $\kP$ parametrizes even more. For instance all hidden Markov measures which are well known to possess infinite degree of dependence in general, see e.g. \cite{blackwell1957entropy,marcus2011entropy}. This, in a certain way,
is in correspondence to an idea presented in \cite{ornstein1973application} which, loosely speaking, describes the truncation of ergodic processes into multi-step Markov processes in terms of their distances measured by entropy. However, the present work gives rise to think of such truncations as suitable projections from a space of ``infinite-step Markov processes''.

Since the  mappings $\psi_{M,N}$ are continuous there is a natural notion of a topology on $\kP$ as the coarsest topology which makes all the mappings $\psi_N$ continuous, see e.g. \cite{bourbaki1968theory,bourbaki1969general} for further details. Taking $\kB$ as the Borel $\sigma$-field induced by this topology makes $\left(\kP,\kB\right)$ a measurable space. Then, basically any probability measure $\Pi\in\cP\left(\cP^e(S^{\N})\right)$ makes $\left(\kP,\kB,\Pi\right)$ a Bayesian model for stationary data taking values in some finite state space. Thereby, the positivity of the parametrized subspace of the ergodic measures amounts to a kind of non-informativity about the possible state-to-state transitions of the process. Basically any information about the inability of some jump into a certain state after having visited a finite string of states would lead the statistical modeling to concentrate on an appropriate subspace of $\kP$. However, positivity usually preserves one from underestimating certain transition probabilities especially when the size of the data is small.


\section{Inference}

In the present section, an explicit model for the mixing measure $\Pi\in\cP(\kP)$ will be obtained. Thereby, the focus shall be on the possibility of a reasonable update to the posterior $\Pi(\cdot|\Xn)$, where $\Xn=(X_1,\dotsc,X_n)$ denotes the observed data. Since the data are assumed to be generally conditionally ergodic, there is uncertainty in any of the summarizing statistics which might be seen as the transition counts of any finite degree in analogy to \eqref{E}. Hence, the update must incorporate the possibility of learning from the data for the degree of dependence of the true data-generating probability measure which will throughout be assumed to be finite. However, it will not be assumed that it is of parametric form in the sense that it is a.s.-$\Pi$ bounded from above as it is for instance in \cite{strelioff2007inferring}. Also it is omitted to resort on purely computational methods as in \cite{bacallado2015finetti} who infer data with respect to their ``closeness to exchangeability''. Therefor, the authors propose the use of a certain prior being motivated from the Central Limit Theorem and exploit Markov Chain Monte Carlo methods to sample from the posterior. This posterior makes then a statement about the order of dependence of the data in terms of its concentration on the respective subspace of the considered parameter space. One might keep in mind the visualization of the embedding of the Bernoulli measures into the space of Markov measures as in above pictures. There the posterior would concentrate around the space diagonal of the parameter cube. However, it is unclear to the author how these methods could be extended to the unbounded case.

The approach considered here is different to the just mentioned ones since it relies on an idea which mimics the update mechanism of the celebrated Dirichlet process. This will enable one to deal with the unbounded case, i.e. where the dependence of the true probability measure generating the data can be of any order. The rough idea is to extend the usual Bayesian sampling scheme by one more stage, namely the drawing of the degree of dependence $N\in\N_0$ from a distribution $\nu\in\cP(\N_0)$. Then, given $N$, a positive stochastic tensor $\kappa_N$ is sampled from a specific distribution $\mu\in \cP(\cK_N)$. This distribution will be taken to reflect another non-informativity, namely that one cannot learn for some segment of the random tensor by observing a transition associated to another segment. This is reflected by sampling the segments of the tensor, which might be thought of as the rows of its flattening analogously as in Example 1, independently from each other. Notice that the ability of any cross learning implicitly required additional prior information on the Bayesian model. Each segment itself is sampled from a Dirichlet distribution $\Dir(\alpha_{t^{(N)}})$ with parameter vector $\alpha_{t^{(N)}}>0$, $t^{(N)}\in S^N$, such that $\mu=\bigotimes_{t^{(N)}\in S^N} \Dir(\alpha_{t^{(N)}})$. By virtue of \thref{thm}, it suffices to sample the stochastic tensor of highest order $\kappa_N$ in order to determine a positive shift-ergodic probability measure $Q=\bigotimes_{M\in\N_0}\psi_{M,N}(\kappa_N)$, where $\psi_{M,N}(\kappa_N)$ amounts to the element $\tilde{\kappa}_M$ of the embedding of $\cK_N$ into $\cK_M$ such that $\phi_{M,N}(\tilde{\kappa}_M)=\kappa_N$ whenever $M>N$. Then, in a last step data is sampled from $Q$ as defined in \thref{thm}. The subsequent sampling scheme \eqref{Sample} summarizes above. Let $\Pi$ be given according to the decomposition
\begin{align*}
\Pi=\int\limits_{\N_0} \left[\bigotimes_{t^{(N)}\in S^N}\Dir(\alpha_{t^{(N)}})\right]\nu(dN)=\sum\limits_{N\in\N_0} \left[\bigotimes_{t^{(N)}\in S^N}\Dir(\alpha_{t^{(N)}})\right]\nu(N)
\end{align*}
such that data is sampled as

\begin{align*}\label{Sample}
&\textsf{(1.) } N\sim \nu\\
&\textsf{(2.) } \kappa_N\big\vert \left[N,(\alpha_{t^{(N)}})_{t^{(N)} \in S^N}\right] \sim \bigotimes_{t^{(N)}\in S^N}\Dir(\alpha_{t^{(N)}})\\ \tag{Smpl}
&\textsf{(3.) }X_{[1:\infty]} \big\vert \left[N,(\alpha_{t^{(N)}})_{t^{(N)}\in S^N},\kappa_N\right]\sim \bigotimes_{N\in\N_0} \kappa_N.
\end{align*}

As the way the data are sampled, the update to the posterior is split in two separate stages, too. Thereby, updating the laws of the stochastic tensors is fairly straight forward. That is basically due to the assumption of independence among the segments of the tensors, which corresponds to the assumption that transition counts of appropriate order are predictively sufficient, c.f. \cite{fortini2014predictive}. Thus, given the order of dependence $N$, the update of the law of $\kappa_N$ is accomplished by separately updating the Dirichlet distributions associated to the segments of the tensor by the respective transition counts. Since the Dirichlet distribution is a conjugate prior for the multinomial likelihoods which emerge from these considerations, the updated law of $\kappa_N$ can be expressed as
\begin{align*}
\kappa_N\big\vert \left[N,(\alpha_{t^{(N)}})_{t^{(N)} \in S^N},\Xn\right] \sim \bigotimes_{t^{(N)} \in S^{N+1}}\Dir\left(\alpha_{t^{(N)}}+\sum\limits_{k=1}^{n-N}\delta_{X_{[k,k+N]}}\left(\{t^{(N)}\}\right)\right).
\end{align*}

In contrast, the update of the distribution of the degree of dependence is not as straight forward as that of the laws of the tensors. That is due to two main reasons. Firstly, one needs to infer the structure of the data from the observations. Secondly, the Bayes theorem implies a procedure which cannot be accomplished by a finite algorithm. For the latter, note that the Bayes theorem would be obtained as
\begin{align*}
\bP\left(N|X_{[1:n]}\right)=\frac{\bP\left(X_{[1:n]}|N\right) \bP(N)}{\bP\left(X_{[1:n]}\right)}
=\frac{\bP\left(X_{[1:n]}|N\right) \nu(N)}{\bP\left(X_{[1:n]}\right)}, \label{B}\tag{B}
\end{align*}
where 
\begin{align*}
\bP(X_{[1:n]}|N)&=\sum\limits_{\kappa_N} \bP(X_{[1:n]}|N,\kappa_N)\bP(\kappa_N|N)\\
&=\sum\limits_{\kappa_N\in \cK_N} \bP(X_{[1:n]}|N,\kappa_N)\left[\bigotimes_{t^{(N)}\in S^N}\Dir(\alpha_{t^{(N)}})\right].
\end{align*}
Hence, the denominator in \eqref{B} is given as
\begin{align*}
\sum\limits_{N\in\N_0}\sum\limits_{\kappa_N\in \cK_N} \bP\left(X_{[1:n]}|N,\kappa_N\right)\left[\bigotimes_{t^{(N)}\in S^N}\Dir\left(\alpha_{t^{(N)}}\right)\right] \nu(N),
\end{align*}
requiring the calculation of the likelihood $\bP(X_{[1:n]}|N,\kappa_N)$ for any $N\in\N_0$ and $\kappa_N\in \cK_N$. From a theoretical point of view this can be done using the family of mappings $\psi$. However, from a practical perspective this seems hopeless.

Hence, another more direct and more pragmatic method is proposed here with particular focus on posterior consistency. In a specific manner, this will consist in mimicking the update mechanism of the Dirichlet process. More precisely this means that the atoms of $\nu$ are updated by weights obtained from the raw data such that the prior guess vanishes by increasing the size of the data. However, there should be an additional mechanism which decreases the mass of atoms which correspond to degrees of dependence that are too large to be inferred from a relatively small size of observations. Moreover, this mechanism should be able to increase such masses when the size of the observations increases and the data give evidence. 

A second difference to the update mechanism of the Dirichlet process is that the additional weights can certainly not stem from an assumption of partially exchangeability of any degree and thus cannot be expressed in terms of the empirical measure. Instead, these updating weights shall be obtained as a kind of empirical measurements of the distances between the tensor spaces in terms of symmetry which are indicated by the observed data. To make this precise, let for $n,N\in\N_0$, $N<n$ and a string $s^{(N)}\in S^{N+1}$
\begin{align*}
\ke_n^{(N)}\left(X_{[1:n]}; s^{(N)}\right)=\frac{1}{n-N} \sum\limits_{k=1}^{n-N} \delta_{X_{[k:k+N]}}\left(s^{(N)}\right)
\end{align*}
be the empirical measure consisting of a properly normalized version of the order-$N$ transition counts $t_N=\left(t_N\left(\Xn; s^{(N)}\right)\right)_{s^{(N)}\in S^{N+1}}$. 
The following thought experiment is given in order to clarify the basic idea of what follows. If e.g. the stochastic process $X$ is known to be sampled from a Markov measure of order one then all the stochastic tensors of order greater than one fulfill a certain symmetry condition representing the truncation of the dependence of the process. That is it holds 
\begin{align*}
p_{s_2\cdots s_{N-1} s_{N}s_{N+1}}^{s_1 s_2\cdots s_{N-1} s_N}=p_{t_2\cdots t_{N-1} s_{N}s_{N+1}}^{t_1 t_2\cdots t_{N-1}s_N},
\end{align*}
for all $\left(s_1,\dotsc,s_{N+1}\right)\in S^{N+1}$, $\left(t_1,\dotsc, t_{N-1}\right)\in S^{N-1}$ and for all $N>1$. When an outcome of process $\Xn$ is observed this means that one will observe nearly as many jumps $\left(s_1,\dotsc,s_{N}\right)\rightarrow \left(s_2,\dotsc,s_{N+1}\right)$ as jumps $\left(t_1,\dotsc,t_{N-1},s_N\right)\rightarrow \left(t_2,\dotsc, t_{N-1},s_N,s_{N+1}\right)$, at least if $n>>N$ is large enough. The argument is the same for orders of dependency $N>1$. From a statistical viewpoint, this argumentation is now reversed and exploited in order to find appropriate update weights for the posterior. Therefor, a suitable estimates for all the transition probabilities is used. These estimates are given for $n>N$ as
\begin{align*}
\hat{p}_{s_2\cdots s_{N-1} s_{N}s_{N+1}}^{s_1 s_2\cdots s_{N-1} s_N}\left(\Xn\right)&=\frac{t_N\left(\Xn; s_1,\dotsc, s_{N+1}\right)}{\sum_{s_N\in S}t_N\left(\Xn; s_1,\dotsc, s_{N+1}\right)}\\
&=\frac{\sum_{k=1}^{n-N}\delta_{X_{[k:k+N]}}\left(s_1,\dotsc, s_{N+1}\right)}{\sum_{k=1}^{n-N}\delta_{X_{[k:k+N-1]}}\left(s_1,\dotsc, s_{N}\right)},
\end{align*}
which is in full analogy of the order one Markov case for which \cite{diaconis1980finetti} argue that $\left(\frac{t_1\left(\Xn; i,j\right)}{\sum_{j\in S}t_1\left(\Xn; i,j\right)}\right)_{i,j\in S}$ converges to a stochastic $(s \times s)$-matrix for almost all sequences of data $X_{[1:\infty]}$ as $n$ tends to infinity. See also \cite{fortini2002mixtures} for related results taking into account a decent extension of the state space. The updating weights will be obtained regarding the differences $\hat{d}_n^{(N)}\left[i,j,s^{(N-1)}\right]$ which are defined as 
\begin{align*}
\bigg\vert \sum_{k=1}^{n-N} & \delta_{X_{[k:k+N-1]}}\left(i,s_1,\dotsc, s_{N-1}\right)\times  \sum_{k=1}^{n-N}  \delta_{X_{[k:k+N]}} \left(j,s_1,\dotsc, s_{N} \right)\\ 
& -\sum_{k=1}^{n-N}\delta_{X_{[k:k+N-1]}}\left(j,s_1,\dotsc, s_{N-1}\right)\times \sum_{k=1}^{n-N}\delta_{X_{[k:k+N]}}\left(i,s_1,\dotsc,s_N \right)\bigg\vert
\end{align*}
for $i,j\in S$ in order to detect dependencies of order $N$. Considering again above thought experiment $\hat{d}_n^{(N)}\left[i,j,s^{(N-1)}\right]$ is expected to decay for all $N>1$ and all strings $s^{(N-1)}$ if $n$ increases, while $\hat{d}_n^{(1)}\left[i,j,s\right]$ is expected to go to infinity for all $s\in S$. Those differences will serve as the building blocks of the update. However, since, using Landaus notation, it holds $\hat{d}_n^{(N)}=\mathcal{O}(n^2)$ for all $N\geq 1$ it is necessary to modify these differences in order to achieve posterior consistency of the update scheme to be presented. Roughly, that is because without such a modification the update scheme won't pin point the tensor space of the highest order as the true one because it treated the lower ones, which do not possess symmetry neither, similarly in an asymptotic manner. The modification of $\hat{d}_n^{(N)}$ will be given as $d_n^{(N)}$ defined through
\begin{align*}
d_n^{(N)}\left[i,j,s^{(N-1)}\right]&:=\bigg\vert  \sum_{k=1}^{n-N}\delta_{X_{[k:k+N-1]}}\left(i,s_1,\dotsc, s_{N-1}\right)  \times \sum_{k=1}^{n-N}  \delta_{X_{[k:k+N]}} \left(j,s_1,\dotsc, s_{N} \right)\\ & -\sum_{k=1}^{n-N}\delta_{X_{[k:k+N-1]}}\left(j,s_1,\dotsc, s_{N-1}\right)\times \sum_{k=1}^{n-N}\delta_{X_{[k:k+N]}}\left(i,s_1,\dotsc,s_N \right)\bigg\vert^N
\end{align*} 

In order to detect the entire symmetry in the stochastic tensors above differences are accumulated in form of $D_n^{(N)}$ defined by
\begin{align*}
D_n^{(N)}=\sum\limits_{\substack{i,j\in S \\ s^{(N-1)} \in S^N }} d_n^{(N)}\left[i,j,s^{(N-1)}\right].
\end{align*}

Now, the update scheme is introduced as follows. The prior $\nu$ reflecting knowledge about the order of dependence is chosen as $\nu(\cdot)=\frac{\beta(\cdot)}{\beta(\N_0)}$, where $\beta$ is a finite measure supported by the non-negative integers. Then define the updated version of the prior after having seen data $\Xn$ by
\begin{align}
\nu\left(N|\Xn\right):= \frac{\beta(N)+D_n^{(N)}}{\beta(\N_0)+\sum_{k\in \N_0}D_n^{(k)} }, \label{U}\tag{U}
\end{align}
where $D_n^{(0)}:=n$ and $D_n^{(N)}:=0$ for $N\geq n$. Then we have the following lemma considering posterior consistency.

\begin{lemma}\thlabel{consistency}
Let $\nu(\cdot)=\frac{\beta(\cdot)}{\beta(\N_0)}$ be a prior for the order of dependence of the data $X$, where $\beta$ is a finite measure on the non-negative integers. Moreover let the true order of dependence of $X$ be $N_0$. Then it holds that 
\begin{align*}
& \nu\left(N_0\vert X_{[1:n]}\right) \stackrel{n \rightarrow \infty}{\longrightarrow} 1,\\
& \nu\left(N\vert X_{[1:n]}\right) \stackrel{n \rightarrow \infty}{\longrightarrow} 1, \textit{\qquad} \forall N\neq N_0,
\end{align*}
where $\nu\left(\cdot | X_{[1:n]}\right)$ is defined as in \eqref{U}
\end{lemma}
\begin{proof}
Since $D_n^{(N)} \in o\left(n^{2N_0}\right)$, the posterior weight $D_n^{(N)}$ vanishes in the limit for all $N> N_0$, while $D_n^{(N)}$ tends to infinity for all $N \geq N_0$. Hence $\nu\left(N\vert X_{[1:n]}\right)$ decays for $N > N_0$. Now, considering $N<N_0$ it follows for large sample sizes
\begin{align*}
\nu\left(N\vert X_{[1:n]}\right) \propto 
&\frac{\beta(N)+D_n^{(N)}}{\beta(\N_0)+D_n^{(1)}+ \dotsc +D_n^{(N_0-1)} + D_n^{(N_0)}}\\
& =\frac{\frac{1}{n^{2N_0}}\left[\beta(N)+D_n^{(N)}\right]}
{\frac{\beta(\N_0)}{n^{2N_0}}+\frac{1}{n^{2(N_0-1)}}\frac{D_n^{(1)}}{n^2}+ \dotsc +\frac{1}{n^2}\frac{D_n^{(N_0-1)}}{n^{2(N_0-1)}} + \frac{D_n^{(N_0)}}{n^{2N_0}} } \stackrel{n \rightarrow \infty}{\longrightarrow} 0.
\end{align*}
since $\frac{D_n^{(N_0)}}{n^{2N_0}}$ attains a non-zero limit. Essentially the same reasoning shows that 
\begin{align*}
\nu\left(N_0\vert X_{[1:n]}\right)\stackrel{n\rightarrow \infty}{\longrightarrow} 1
\end{align*}
such that the lemma is proven.
\end{proof}

\thref{consistency} along with well known results about posterior consistency of Dirichlet priors do imply the following result.

\begin{corollary}
Let $\nu\left(\cdot\right)=\frac{\beta(\cdot)}{\beta(\N_0)}$ be a prior for the order of dependence, where $\beta$ is a finite measure on the non-negative integers. Moreover, given the order of dependence, let the rows of the flattened stochastic tensor representing this order be sampled independently according to Dirichlet priors $\Dir\left(\alpha_{t^{(N)}}\right)$ with. Then, for almost all sequences of data $X_{[1:\infty]}$, the sequence of posterior measures defined by 
\begin{align*}
\Pi\left(\cdot\vert X_{[1:n]}\right)= \sum\limits_{N\in \N_0} \left[\bigotimes_{t^{(N)}\in S^{N+1}} \Dir \left(\alpha_{t^{(N)}}\big\vert \Xn\right)\right] \nu\left(N\vert \Xn\right)
\end{align*}
centers around the true data generating measure which samples data according to scheme \eqref{Sample}.
\end{corollary}


\section{Remarks}

The wish of having a better understanding about the foundations of Bayesian statistics for general stationary data arose from thinking from a Bayesian perspective about a problem concerning the $M/G/\infty$ queue which came up in \cite{brown1970m}. Therein, the author studies an estimation problem in queueing theory from the frequentist viewpoint. By the statistical setup of this paper, one has only access to general stationary and, according to the frequentist interpretation of probability, ergodic data. To approach this issue, Brown pretends as the data was i.i.d. and gives an estimator which resembles the empirical distribution function. However, one can find some words of criticism in his work which reads as
\begin{quote}
Although we have obtained an estimator for the c.d.f. G, it is clearly not the best estimator in any sense because we do not use all the information. The problem of finding a best estimator (according to any criterion) is still open.
\end{quote}
The present work gives an idea of what the suppressed information is as well as what an improved estimator might incorporate. However, in order to overcome rather non-telling technicalities, it was necessary to break down the problem into a discrete-time setting as well as to a finite state space. However, the questions about generalizing to a more general state spaces and to a continuous-time domain, respectively, are still open.

 It seems possible to obtain a similar construction exploiting theory of Markov chains on countably infinite state spaces. It is well known that for irreducible and positive recurrent stochastic matrices there exists a unique invariant distribution which might serve as the building blocks of a generalization of the theory just presented. However, from a theoretical viewpoint it seems more difficult to obtain a measurable structure on the generalized inverse limit construction since continuity would have to be achieved by studying perturbations of linear operators, see e.g. \cite{kato1966perturbation}. From a rather applied viewpoint the question is how the system of non-linear equations \eqref{N} can be reasonably approximated in case of countably infinite state spaces. 

The explicit model which was introduced in the present work represents a specific sampling scheme which bases on several independence assumptions. These assumptions, in a certain sense, reflect the lack of information about the possibility of cross-learning for the transitions of the observed process. At this point the model might be further specified if one has gained additional information about this cross-learning in advance. Notice that such information can be given by certainty of the specific process one does collect data from, which in turn shrinks the class of models. The $M/G/1$ queueing model might serve as an example for the possibility of cross-inference which even takes into account a countably infinite state space. In queueing theory it is well known that the $M/G/1$ queue is driven by an embedded Markov chain governed by a so called stochastic $\Delta$-matrix, c.f. \cite{abolnikov1991markov}. Observing data that stems from an $M/G/1$ queue, one can make cross-inference based on the specific shape of these stochastic matrices. For more details see also \cite{rohrscheidt2017service}. Considering the more general case of an arbitrary infinite stochastic matrix techniques of \cite{gibson1987augmented,gibson1987monotone} might be of interest in order to obtain Bayesian models for making inference for Markov chains (possibly of higher order) on countably infinite state spaces.

The question about generalizations with respect to some uncountable state space is an even deeper one. Therefor, the work of \cite{fortini2002mixtures} may give a first idea how complicated things become considering general Markov kernels and how overwhelmingly large the possible class of Bayesian models gets. More precisely, again the question of cross-learning for transitions come up here but in a even more involved manner. This is because one would need to clarify within the modeling how much one can learn for transitions ''being close´´ to observed ones. One might think about a stone being thrown onto a surface of some fluid. The question is then how fast do the waves, representing the ''strength of cross-learning´´, have to decay? The fact that the probability of an observation of any of such transitions is zero doesn't make it any easier.

Another difficult question is about Bayesian inference for processes evolving in continuous time. The first work on mixtures of additive processes is \cite{buhlmann1960austauschbare} followed by \cite{freedman1963invariants, freedman1996finetti} which additionally study mixtures of continuous-time Markov processes. Further approaches to symmetries of continuous-time processes can be found in \cite{kallenberg2006probabilistic}. Although the results obtained therein might be understood as the probabilistic fundament for Bayesian statistics for continuous-time stochastic processes, there is still a lack of decently working statistical models. 
One reason of that fact might be that collecting data from a continuous-time stochastic process at fixed discrete time points would lead to a rather discrete-time model in advance. On the other hand, collecting data in a high-frequency manner, similarly to the consideration of an uncountable state space, might cause cross-learning effects 
which would have to be incorporated in the statistical model as prior information.\\

\textbf{Acknowledgments:} The present work is an excerpt of my doctoral thesis \textit{Bayesian Nonparametric Inference for Queueing Systems} which I prepared at the Ruprecht-Karls-Universit\"{a}t Heidelberg within the DFG-project (German research foundation) grant 1953. I would like to thank my supervisors Prof. Rainer Dahlhaus and Dr. Cornelia Wichelhaus for their patience in letting me unfold the problem. I am also grateful to Prof. Jan Johannes who helped to improve the work in numerous discussions.

\bibliographystyle{imsart-nameyear}
\bibliography{mybibB4S}

\end{document}